\documentclass[12pt,reqno]{amsart}

\usepackage{amssymb,amsthm}

\def\oo#1{{}_{(#1)}}
\def\Cur{\mathop {\fam 0 Cur}\nolimits}
\def\End{\mathop {\fam 0 End}\nolimits}
\def\Cend{\mathop {\fam 0 Cend}\nolimits}
\def\Conf{\mathop {\fam 0 Conf}\nolimits}
\def\Pois{\mathop {\fam 0 Pois}\nolimits}
\def\Lie{\mathop {\fam 0 Lie}\nolimits}

\newtheorem{proposition}{Proposition}
\newtheorem{theorem}{Theorem}
\newtheorem{corollary}{Corollary}

\title[Poisson algebras as conformal modules]{Defining relations
and Gr\"obner--Shirshov bases of Poisson algebras 
as~of~conformal modules}
\author{P. S. Kolesnikov, A. S. Panasenko}
\address{Sobolev Institute of Mathematics, Novosibirsk, Russia\\
Novosibirsk State University, Novosibirsk, Russia}

\thanks{The work is supported by Mathematical Center in Akademgorodok}

\begin{document}

\begin{abstract}
We study the relation between Poisson algebras and representations 
of Lie conformal algebras. We establish a setting for the 
calculation of a Gr\"obner--Shirshov basis in a module 
over an associative conformal algebra and apply this technique 
to Poisson algebras considered as conformal modules over appropriate 
associative envelopes of current Lie conformal algebras. As a result, 
we obtain a setting for the calculation of a Gr\"obner--Shirshov basis 
in a Poisson algebra.
\end{abstract}

\maketitle

\section{Introduction}

Conformal algebras (also known as Lie vertex algebras) appear in 
the theory of vertex operator algebras as a formal language describing 
the singular part of the operator product expansion (OPE) \cite{Kac1998}. 
The entire OPE of vertex operators may be completely recovered 
from the singular part by adding the only operation of normally ordered 
product which is known to be left-symmetric \cite{BK2002}.

In this note, we consider a relation between Poisson algebras and 
representations of conformal algebras. As a possible application, the 
equality problem is the class of Poisson algebras reduces to 
the equality problem for modules over associative conformal algebras.

Suppose $P$ is a Poisson algebra over a field $\Bbbk $ of characteristic zero
and $L$ is a Lie subalgebra of $P$.
That is, $P$ is a commutative algebra equipped with a Lie bracket $\{\cdot,\cdot\}$
satisfying the Leibniz rule 
\[
\{x,yz\} = y\{x,z\} +z\{x,y\},\quad x,y,z\in P.
\]
For example, if $L$ is a Lie algebra then the
associated graded algebra $\mathrm{gr}\,U(L)$ of the 
 universal enveloping associative algebra $U(L)$
is a Poisson algebra denoted $P(L)$, the bracket extends 
the Lie product on~$L$.

For $a\in L$, $u\in P$, denote by $(a\oo_{\lambda } u)$ the polynomial 
$\{a,u\}+\lambda au$ in a formal variable $\lambda $ with coefficients in~$P$.
Denote $H=\Bbbk [\partial ]$, and extend the operation $(\cdot \oo{\lambda } \cdot )$
to the free $H$-modules $H\otimes L$, $H\otimes P$ as follows:
\[
 (f(\partial )a\oo{\lambda } g(\partial )u ) = f(-\lambda )g(\partial +\lambda ) (a\oo\lambda u).
\]
This is straightforward to compute \cite{Kol2020} that 
\begin{equation}\label{eq:ReprLieJacobi}
(x\oo\lambda (y\oo\mu w)) - (y\oo\mu (x\oo\lambda w)) = [x\oo\lambda y]\oo{\lambda +\mu} w
\end{equation}
for $x,y\in H\otimes L$, $w\in H\otimes P$, where $[x\oo\lambda y]$ in the right-hand side 
is the conformal Lie bracket in the current conformal algebra 
$\Cur L$.

Therefore, in particular, given a Lie algebra $L$ the free $H$-module generated by
the Poisson enveloping algebra $P(L)$ is a conformal module over $\Cur L$.
As in the case of ordinary Lie algebras, a conformal module structure over a conformal algebra
gives rise to a module structure over its universal associative conformal envelope \cite{Ro2000}. 
In contrast to ordinary algebras, the universal enveloping associative conformal algebra for a given 
conformal algebra is not unique, and the choice of appropriate one is determined by the
locality function on the conformal linear maps corresponding to the particular representation.

In our case, consider the conformal linear maps
$\rho(a)_\lambda  \in \Cend (H\otimes P)$, $a\in L$, sending 
$w\in P$ to $(a\oo\lambda w) = \{a,w\} + \lambda aw$.
Evaluate the $\lambda $-product of two such maps:
\begin{multline*}
(\rho(a) \oo\lambda \rho(b))_\mu  = \rho(a)_\lambda (\rho(b)_{\mu-\lambda } w) \\
= \{a,\{b,w\}\} +\lambda a \{b,w\} +(\mu-\lambda ) \{a, bw\} + \lambda (\mu-\lambda )abw.
\end{multline*}
The result is a quadratic polynomial in $\lambda $, so the locality level of $\rho(a)$ and $\rho(b)$
is $N=3$. Hence, $H\otimes P$ is a conformal module over $U(\Cur L, N=3)$. The latter 
associative conformal algebra was studied in \cite{KK2020}. 

Given a Lie algebra $L$, 
we find the explicit set of defining relation of $H\otimes P(L)$ as of conformal $\Cur L$-module.
In particular, if $L$ is the free Lie algebra $\Lie (X)$ generated by a set $X$ then $P(L)$ is the 
free Poisson algebra $\Pois (X)$. For a set $S\subset \Pois (X)$, denote by $I$ the ideal in $\Pois (X)$ generated by $S$. Then $H\otimes I\subset H\otimes \Pois (X)$ is a conformal submodule 
over $\Cur \Lie (X)$. Therefore, finding a Gr\"obner--Shirshov basis in the conformal module 
allows solving the equality problem in $\Pois (X)$. This is an alternative approach relative to 
the theory of Gr\"obner--Shirshov bases in Poisson algebras \cite{BCZ19}.

\section{Conformal endomorphisms and universal associative envelopes}

Let $L$ be a (Lie) conformal algebra, i.e., a left $H$-module equipped with a $\lambda $-bracket 
\[
[\cdot\oo\lambda \cdot ] : L\otimes L \to H[\lambda ]\otimes_H L
\]
which is sesqui-linear 
and satisfies skew-commutativity 
\[
[x\oo\lambda y] = -[y\oo{-\lambda-\partial} x], \quad x,y\in L
\]
and Jacobi identity 
\[
[ x \oo\lambda [y\oo\mu z]] - [y\oo\mu [ x \oo\lambda z]] = [ [x\oo\lambda y]\oo{\lambda+\mu } z ],
\quad
x,y,z\in L.
\]

A representation of $L$ on a left $H$-module $M$ is defined by sesqui-linear $\lambda $-operation 
\[
(\cdot \oo\lambda \cdot ): L\otimes M \to H[\lambda ]\otimes _H M
\]
which satisfies \eqref{eq:ReprLieJacobi} for all $x,y\in L$, $w\in M$.

In other words, a representation is a map sending $x\in L$ to the operation 
$\rho(x) = (x\oo\lambda \cdot ): M\to H[\lambda ]\otimes _H M$. 
The operations $\rho(x)$, $x\in L$, are conformal linear operators on $M$ in the sense of 
\cite{Kac1998}, i.e., $\rho : L\to \Cend M$. Recall that $\Cend M$ has a $\lambda $-operation 
$(\cdot \oo\lambda \cdot )$ which has not necessarily polynomial values:
if 
$\rho(x) = (x\oo\lambda \cdot )$, $\rho(y) = (y\oo\lambda \cdot )$ then 
\[
\begin{aligned}
(\rho(x)\oo\lambda \rho(y)):{}&  M \to H[[\lambda ]] [\mu] \otimes _H M, \\
 & w\mapsto (x\oo\lambda (y\oo{\mu-\lambda} w)),
\quad w\in M.
\end{aligned}
\]
If for every $x,y\in L$ the image of $(\rho(x)\oo\lambda \rho(y))$ is a polynomial in $\lambda $
(in particular, this is the case when $M$ is a finitely generated $H$-module) then 
by the Dong's Lemma (see, e.g., \cite{Kac1998}) $\rho(L)$ generates an associative 
conformal subalgebra $E=E(L,\rho)$ in $\Cend M$. 

Denote by $E^{(-)}$ the commutator Lie conformal algebra based on $E$:
\[
 [f\oo\lambda g] = (f\oo\lambda g) - (g\oo{-\lambda -\partial } f),\quad f,g\in E.
\]
The relation \eqref{eq:ReprLieJacobi} and the definition of  $(\rho(x)\oo\lambda \rho(y))$
imply $[\rho(x)\oo\lambda \rho(y)] = \rho([ x \oo\lambda y])$ for $x,y\in L$.
Therefore, $E$ is an associative envelope of $L$. 

Every associative envelope of a Lie conformal algebra is an image of an appropriate 
universal enveloping  associative conformal algebra \cite{Ro2000}. 
The choice of the universal envelope is determined by the degrees of 
the polynomials
$(\rho(x)\oo\lambda \rho(y))$, $x,y\in L$. If $\deg_\lambda (\rho(x)\oo\lambda \rho(y)) <N(x,y)$ 
then $E =E(L,\rho )$  is an image of the associative conformal algebra $U(L,N)$.

As $M$  is a conformal module over $L$ then the universal property of $U(L,N)$ 
implies that there exists a representation of $U(L,N)$ on $M$ extending $\rho $. 
Note that the canonical map $L\to U(L,N)$ is not necessarily injective.

For example, if $L=\Cur\mathfrak g$ is the current Lie conformal algebra over a Lie algebra $\mathfrak g$
and $M=L$ is the regular conformal $L$-module then 
$(\rho(a)\oo\lambda \rho(b)) = \mathrm{ad}\, a \mathrm{ad}\, b \in \End \mathfrak g$, 
$a,b\in \mathfrak g$,
so $N(a,b)=1$. The corresponding universal enveloping associative conformal algebra 
$U(\Cur\mathfrak g, N=1)$ 
is just 
$\Cur U(\mathfrak g)\mathfrak g$.

For the same $L$, if $M=H\otimes P(\mathfrak g)$, and $(a\oo\lambda w) = [a,w] + \lambda aw$, $a\in \mathfrak g$, 
$w\in P(\mathfrak g)$, then $N(a,b)\le 3$ for $a,b\in \mathfrak g$. Hence, the corresponding envelope 
is an image of $U(\Cur \mathfrak g, N=3)$.
Hence, 
$H\otimes P(\mathfrak g)$ is a conformal module over the associative conformal algebra $U(\Cur \mathfrak g, N=3)$. Obviously, it is generated by the single element $1=1\otimes 1$. The problem addressed in this paper is to determine 
the defining relations and find a complete (confluent) set of rewriting rules for this conformal module.

\section{Gr\"obner--Shirshov bases of conformal modules}

Let $C$ be a conformal algebra generated by a set $X$. Then $C$ is an image of an appropriate 
free associative conformal algebra \cite{Ro1999} generated by $X$ relative to a locality 
function $N: X\times X\to \mathbb Z_+$. Denote this free system by $\Conf(X,N)$. The latter may be presented as follows \cite{Ko2020}.

 Denote by $A(X)$ the ``ordinary'' associative algebra generated by the set 
\[
 B(X)=\{\partial \}\cup \{L_n^a,R_n^a \mid a\in X,\,n\in \mathbb Z_+\}
\]
with the defining relations
\begin{align}\label{eq:rel_A(X)1}
 & L_n^a\partial - \partial L_n^a -n L_{n-1}^a,  \\
 & R_n^a\partial - \partial R_n^a -n R_{n-1}^a,  \label{eq:rel_A(X)2} \\
 & R_n^aL_m^b - L_m^b R_n^a, \label{eq:rel_A(X)3}
\end{align}
For a given function $N:X\times X\to \mathbb Z_+$ 
consider the left $A(X)$-module $M(X,N)$
generated by $X$ with the following relations:
\begin{align}\label{eq:mod_A(X)1}
  & L_n^ab, \quad n\ge N(a,b),\\
  & R_m^ba - \sum\limits_{s= 
0}^{N(a,b)-m}(-1)^{m+s}\frac{1}{s!}\partial^{s}L_{m+s}^a b,
  \quad m\in \mathbb Z_+, \label{eq:mod_A(X)2}
\end{align}
where $a,b\in X$.

\begin{proposition}[\cite{Ko2020}]\label{cor:FreeBasis_module}
The free associative conformal algebra $\Conf(X,N)$ is an 
an $A(X)$-module relative to 
\[
\begin{gathered}
(a\oo\lambda u) = \sum\limits_{n\ge 0} \dfrac{\lambda^n}{n!} L_n^a u,
\quad 
(u\oo{-\lambda-\partial} a) = \sum\limits_{n\ge 0} \dfrac{\lambda^n}{n!} R_n^a u, \\
\quad u\in \Conf(X,N),\ a\in X, n\ge 0.
\end{gathered}
\]
Moreover,  $\Conf(X,N)$ 
and $M(X,N)$ are isomorphic as $A(X)$-modules.
\end{proposition}

A relation in $\Conf (X,N)$ may be rewritten as a relation in the free $A(X)$-module generated by $X$. 
Finding a Gr\"obner--Shirshov basis for such a module \cite{KangLee2000} 
is the same as finding the Gr\"obner--Shirshov 
basis for an associative conformal algebra. Hence, every associative conformal algebra $C$ may be 
presented as a quotient of $M(X,N)$ with a Gr\"obner--Shirshov basis 
$S\subset A(X)\otimes \Bbbk X$.

Now, consider a (left) conformal module $M$ over $C$ generated by a set $Y$
relative to a locality function $N': X\times Y\to \mathbb Z_+$.
The defining relations $T$ of $M$ may be written as elements of the free $A(X)$-module 
generated by $Y$ in the same way as it is done with the relations of~$C$. 

The following statement describes the setting for Gr\"obner--Shirshov bases computation 
in $M$ which is less technical than proposed in \cite{ChenLu2020}. Moreover, our technique is based on 
the ``ordinary'' associative algebra and thus the computations may be performed with either of existing 
computer algebra packages.

Let $C$ be an associative conformal algebra generated by a set $X$ (relative to a locality function $N$) with defining relations 
$S\subset A(X)\otimes \Bbbk X$. 
Suppose $M$ is a left conformal $C$-module generated by a set $Y$ (relative to a locality function $N'$) with defining relations $T\subset A(X)\otimes \Bbbk Y$.
Then the split null extension $C\oplus M$ is an associative conformal algebra 
generated by $X\cup Y$ relative to the locality function 
$N$ that extends the locality on $X\times X$ and $X\times Y$
with $N(Y,X)=N(Y,Y)=0$.

\begin{theorem}\label{thm:GSB-mod}
The defining relations of $C\oplus M$ are 
$S\cup T$ along with 
$L_n^y u$, $y\in Y$, $n\ge 0$, $u\in A(X\cup Y)\otimes \Bbbk(X\cup Y)$.
\end{theorem}
 
\begin{proof} 
Recall that $(M\oo\lambda M) = (M\oo\lambda C) = 0$ in $C\oplus M$.
Note that \eqref{eq:rel_A(X)1}--\eqref{eq:rel_A(X)3} 
is a Gr\"obner--Shirshov basis, so  $A(X)\subset A(X\cup Y)$
and $A(Y)\subset A(X\cup Y)$.

It is sufficient to prove that an arbitrary 
element of $A(X\cup Y)\otimes (\Bbbk X \oplus \Bbbk Y)$ 
is equivalent modulo the relations mentioned in the statement either to 
an element of $A(X)\otimes \Bbbk X$ or to an element of $A(X)\otimes \Bbbk Y$.

First, assume $f \in A(X,Y)\otimes \Bbbk Y$, $f = u\otimes y$, $u$ is a monomial in $A(X\cup Y)$. 
Without loss of generality we may suppose 
\[
 u = \partial ^s L_{n_1}^{x_1}\dots L_{n_k}^{x_k} R_{m_1}^{z_1}\dots R_{m_p}^{z_p},
\]
where $x_i\in X$, $z_j\in X\cup Y$. 
If $z_p\in X$ then the image of $u\otimes y $ is zero due to the relation
\eqref{eq:mod_A(X)2} which allows rewriting
 $R_n^xy$ via $L_n^yx$.
 If $z_p\in Y$ then one may apply the relation in \eqref{eq:mod_A(X)2} and obtain zero again due to $L_n^y$.
Hence, $u$ contains only $L_{n_i}^{x_i}\in A(X)$ as required.

Next, assume  $f = u\otimes x$, $u$ is a monomial in $A(X\cup Y)$, $x\in X$. 
Present $u$ in the same form as above and choose maximal $j=1,\dots, p$ such that 
$z_j\in Y$. (If there is no such $j$ then $f$ belongs to $A(X)\otimes \Bbbk X$.)
Then $R_{m_{j+1}}^{z_{j+1}}\dots R_{m_p}^{z_p}x $ may be rewritten in terms of $L_n^z$, $z\in X$, $n\ge 0$ via  \eqref{eq:mod_A(X)2}. One may interchange $R_{m_j}^{z_j}$ 
with all these operators via \eqref{eq:rel_A(X)3} and thus reduce $f$ to a linear combination 
of words ending with $R_m^{z_j} z$, $z\in X$. The latter rewite to the expressions ending with $z_j$.
If there exists at least one more $z_l \in Y$ ($l=1,\dots, j-1$) then $f$ is equivalent to zero as shown in the previous paragraph. Otherwise, $f$ is equivalent to an expression from  $A(X)\otimes \Bbbk Y$ as required.
\end{proof}

\begin{corollary}\label{cor:GSB_mod}
A Gr\"obner--Shirshov basis of a conformal module $M$ over an associative conformal algebra $C$ 
consists of all those relations from a Gr\"obner--Shirshov basis of $C\oplus M$ that 
contain a single letter $y\in Y$.
\end{corollary}

In order to construct a Gr\"obner--Shirshov basis of $M$ according to Corollary \eqref{cor:GSB_mod} one should 
proceed as follows. 

Define the associative algebra $A$ generated by 
$B(X\cup Y)\setminus \{L_n^y \mid y\in Y, n\in \mathbb Z_+\}$
with respect to the defining relations 
\eqref{eq:rel_A(X)1}--\eqref{eq:rel_A(X)3}
(for $a\in X\cup Y$). 
The reason is that we do not need $L_n^y$ which is identically zero operator, but we need $R_n^y$ for $y\in Y$.
Consider the free $A$-module 
generated by $X\cup Y$ relative to the defining relations 
\eqref{eq:mod_A(X)1}--\eqref{eq:mod_A(X)2} 
(for $a\in X$, $b\in X\cup Y$).
Add the defining relations $S\cup T$ and find a Gr\"obner--Shirshov basis of the obtained set 
of relations. Finally, choose those relations that contain 
a single letter from $Y$.

\section{Poisson enveloping algebras as conformal modules}

Let us apply the technique mentioned above to find a
Gr\"obner--Shirshov basis of $H\otimes P(\mathfrak g)$ as of 
a conformal module over the associative conformal algebra $U=U(\Cur\mathfrak g, N=3)$.
As a necessary part, we need a Gr\"obner--Shirshov basis of $U$ found in \cite{KK2020}. 

Let us fix a linear basis $X$ of $\mathfrak g$, and let $Y=\{1\}$. Assume the set $X\cup Y$ is linearly ordered in such a way that $1>X$. We will denote $R_n^1$ simply by $R_n$.

For associative envelopes of Lie conformal algebras, the algebra $A$ may be slightly modified
by adding the family of commutation relations on $L_n^x$.
Namely, let $A$ be an associative algebra generated by the set 
\[
B = \{\partial, R_n, L_n^a, R_n^a \mid n\ge 0, a\in X \}
\]
relative to the following defining relations written as rewriting rules:
\begin{equation}\label{eq:U3-rulesA}
 \begin{gathered}
 \partial L_0^a \to L_0^a\partial ,
 \quad   \partial L_1^a\to L_1^a\partial -L_0^a, \\
 L_n^a\partial \to \partial L_n^a + nL_{n-1}^a,\quad n\ge 2, \\
 R_n^a\partial \to \partial R_n^a + nR_{n-1}^a,\quad n\ge 0, \\
 R_n \partial \to \partial R_n + nR_{n-1},\quad n\ge 0, \\ 
 R_m^aL_n^b \to L_n^bR_m^a,\quad n,m\ge 0, \\
 R_m L_n^b \to L_n^bR_m ,\quad n,m\ge 0, \\
 L_n^aL_m^b \to L_m^bL_n^a + L_{n+m}^{[a,b]},\quad (n,a)>_{lex}(m,b);
 \end{gathered}
\end{equation}
Consider a left $A$-module generated by the set $X\cup \{1\}$ 
with defining relations
\begin{equation}\label{eq:U3-rulesB}
 \begin{gathered}
L_n^ab\to 0,\quad R_n^ab\to 0,\quad n\ge 3, \\
L_0^a 1\to 0,\quad L_n^a  1 \to 0,\quad n\ge 2, \\
R_n^a 1\to 0, \quad R_n1\to 0,\quad n\ge 0, \\
R_2^ab \to L_2^ab,\quad R_1^ab\to L_1^{a}b , \\
R_2a \to 0, \quad R_1a\to -L_1^a1,\quad R_0a\to -\partial L_1^a1, \\
R_0^ab \to L_0^ab - [a,b];
 \end{gathered}
\end{equation}
\begin{equation}\label{eq:U3-rulesC}
 \begin{gathered}
 L_2^ab \to L_2^ba,\quad a>b, \\
 \partial L_2^ab \to L_1^ab +L_1^ba , \\
 L_1^a\partial b \to L_1^b\partial a + 3L_0^ab-3L_0^ba -2[a,b],
 \quad a>b.
 \end{gathered}
\end{equation}

In order to get a Gr\"obner--Shirshov basis of $U(\Cur\mathfrak g, N=3)$
it is enough to choose the relations without $R_m$ or $1$ and add the following 
rewriting rules \cite{KK2020}.

\begin{equation}\label{eq:GSB-Ds}
\begin{gathered}
L_1^a\partial^sb \to L_1^b\partial^sa - (s+2)L_0^b\partial^{s-1} a + 
(s+2)L_0^a\partial^{s-1}b - 2\partial^{s-1}[a, b], \\
\quad s\ge2,\ a>b,
\end{gathered}
\end{equation}
\begin{gather}
L_2^aL_2^bc \to 0, \quad a,b,c\in X_1           \label{eq:GSB-2.2} 
\\
L_1^aL_2^bc \to L_1^bL_2^ca ,\quad b\le c<a,         \label{eq:GSB-1.2}
\\
L_1^aL_2^bc \to L_1^bL_2^ac ,\quad b<a\le c,          \label{eq:GSB-1.2'} 
\end{gather}
\begin{multline}
L_1^aL_1^bc \to L_1^c L_1^a b +  L_0^b L_2^c a -  L_0^c L_2^a b + 
L_2^c [a,b] + L_2^a [c, b],\\
c<a\le b,              
                             \label{eq:GSB-1.1'}
\end{multline}
\begin{multline}
L_1^aL_1^bc \to  L_1^a L_1^c b +  L_0^b L_2^a c - L_0^c L_2^a b + 
L_2^a [c, b] \\
+ L_2^b [c,a] + L_2^c [a,b],\quad a\le c<b,    \label{eq:GSB-1.1}
\end{multline}
\begin{multline}
L_0^aL_1^bc \to L_0^a L_1^c b +  L_0^b L_1^a c + L_0^c L_1^b a -
L_0^b L_1^c a - L_0^c L_1^a b + L_1^{[c,a]} b +\\+ L_1^{[a,b]} c + 
L_1^{[b,c]} a  - L_1^c [a,b] - L_1^a [b,c] - L_1^b [c,a] ,
\quad c<b<a,         \label{eq:GSB-0.1}
\end{multline}

\begin{theorem}\label{thm:PBW}
The rules 
\eqref{eq:U3-rulesA}--\eqref{eq:GSB-0.1}
form a set  of defining relations for the split null extension
$E = U(\Cur\mathfrak g, N=3)\oplus (H\otimes P(\mathfrak g))$.
To construct a Gr\"obner--Shirshov basis, it is enough to add
the rules
\[
L_0^a\partial^s 1 \to 0,\quad a\in X, \ s>0,
\]
and
\begin{equation}\label{eq:Leibniz}
L_0^a L_1^{b_1}\dots L_1^{b_k} \partial ^s 1 \to \sum\limits_{i=1}^k L_1^{b_1}\dots L_1^{[a,b_i]} \dots L_1^{b_k} \partial ^s1
\end{equation}
for $a,b_1,\dots ,b_k\in X$,
$b_1\le \dots \le b_k$, $s\ge 0$, $k\ge 1$.
\end{theorem}

Hence, the specific defining relations of $H\otimes P(\mathfrak g)$
as of a conformal module over $U(\Cur \mathfrak g, N=3)$
are 
\[
L_n^a1\to 0,\quad n=0,2,3,\dots .
\]
\begin{proof}
Let us start with the intersection of 
the first relation in \eqref{eq:U3-rulesC}
with $R_0\partial \to \partial R_0 $.
On the one hand,
\begin{multline*}
    R_0\partial L_2^a b \to \partial R_0L_2^ab 
\to \partial L_2^aR_0b  
\to -\partial L_2^a\partial L_1^b 1 
\to -\partial L_2^aL_1^b \partial 1 + \partial L_2^aL_0^b 1
\\
\to -\partial L_1^bL_2^a\partial 1 -\partial L_3^{[a,b]}\partial 1 + \partial L_0^bL_2^a1 +\partial L_3^{[a,b]}1 \to -2\partial L_1^bL_1^a 1 
\\
\to -2L_1^bL_1^a\partial 1 + 2L_0^bL_1^a 1 + 2L_1^bL_0^a 1
\to -2L_1^bL_1^a\partial 1 + 2L_0^bL_1^a 1.
\end{multline*}
On the other hand, 
\begin{multline*}
R_0\partial L_1^a b + R_0L_1^ba 
\to -L_1^a\partial L_1^b1 - L_1^b\partial L_1^a 1 
\to -L_1^a L_1^b \partial 1 - L_1^b L_1^a \partial \\ 1
+ L_1^aL_0^b1 + L_1^bL_0^a1
\to -2L_1^bL_1^a\partial 1 - L_2^{[a,b]} \partial 1
\to -2L_1^bL_1^a\partial 1 - 2L_1^{[a,b]}1
\end{multline*}
Therefore, the composition is 
\[
L_0^bL_1^a 1 + L_1^{[a,b]} 1
\]
which is equivalent to the rewriting rule 
\[
L_0^bL_1^a 1 \to L_1^{[b,a]} 1,\quad a,b\in X.
\]
The latter is exactly \eqref{eq:Leibniz}
for $k=1$.

Proceed by induction on $k$. Assume \eqref{eq:Leibniz}
holds for some $k$ (with $s=0$), then the composition of intersection with $L_1^{b_{k+1}}L_0^a \to L_0^aL_1^{b_{k+1}} + L_1^{[b_{k+1},a]}$
leads to the same sort rule for $k+1$.

The compositions of the rule \eqref{eq:Leibniz} with 
$\partial L_0^a\to L_0^a\partial $
gives rise to the desired relations for $s\ge 1$.

It is straightforward to check that the remaining compositions are all trivial.
For example, let us consider the intersection of \eqref{eq:GSB-0.1}
with $R_1L_0^a\to L_0^aR_1$. On the one hand, 
\[
R_1L_0^aL_1^bc \to L_0^aL_1^bR_1c \to -L_0^aL_1^bL_1^c1 
\to -L_1^{[a,b]}L_1^c 1 - L_1^bL_1^{[a,c]} 1.
\]
On the other hand, 
\begin{multline*}
R_1L_0^aL_1^bc \to 
L_0^a L_1^c R_1b +  L_0^b L_1^a R_1c + L_0^c L_1^b R_1a -
L_0^b L_1^c R_1 a - L_0^c L_1^a R_1b \\
+ L_1^{[c,a]} R_1b 
+ L_1^{[a,b]} R_1 c + 
L_1^{[b,c]} R_1a  - L_1^c R_1[a,b] - L_1^a R_1[b,c] - L_1^b R_1[c,a] \\
\to 
-L_0^a L_1^c L_1^b 1 -  L_0^b L_1^a L_1^c1 - L_0^c L_1^b L_1^a1
+L_0^b L_1^c L_1^a 1 + L_0^c L_1^a L_1^b 1 - L_1^{[c,a]} L_1^b 1 \\
- L_1^{[a,b]} L_1^c 1 -L_1^{[b,c]} L_1^a 1  + L_1^c L_1^{[a,b]} 1 
+ L_1^a L_1^{[b,c]} 1 + L_1^b L_1^{[c,a]} 1
 \\
=
\big( -L_0^a L_1^c L_1^b 1 + L_1^{[a,c]} L_1^b 1 + L_1^c L_1^{[a,b]} 1 \big)
+\big (-L_0^b L_1^a L_1^c1 + L_1^{[b,a]} L_1^c 1 \\
+ L_1^a L_1^{[b,c]} 1 \big)
+ \big( L_0^c L_1^a L_1^b 1 -L_0^c L_1^b L_1^a1 \big)
+ \big (L_0^b L_1^c L_1^a 1 -L_1^{[b,c]} L_1^a 1  \big)
+ L_1^b L_1^{[c,a]} 1 \\
\to 
L_0^c L_2^{[a,b]} 1 + L_1^c L_1^{[b,a]} 1 + L_1^bL_1^{[c,a]} 1.
\end{multline*}
The last to expressions are equal modulo the relations 
$L_2^x1\to 0$, $x\in \mathfrak g$.

\medskip

Note that: 

\[ L_2^aL_1^b 1 \to L_1^bL_2^a 1 + L_3^{[a,b]}1 \to 0\]

\[ L_n^a \partial 1 \to \partial L_n^a 1 + n L_{n-1}^a 1 \to 0, \qquad n\ge 3\]

We will use these last two relations without explanations.

Let us check \eqref{eq:U3-rulesC}. First relation, multiplied by $R_0$, $a>b$:

\begin{multline*} R_0L_2^a b \to - L_2^a \partial L_1^b 1 \to - L_2^aL_1^b \partial 1 + L_2^aL_0^b1\to -L_1^bL_2^a\partial 1 - L_3^{[a,b]}\partial 1 \\ \to -L_1^bL_2^a\partial 1 - \partial L_3^{[a,b]}1 - 3L_2^{[a,b]}1 \to -L_1^bL_2^a\partial 1 \to -L_1^b\partial L_2^a 1 - 2 L_1^bL_1^a 1 \\\to -2 L_1^bL_1^a 1\end{multline*}

\[ R_0 L_2^b a \to \dots - 2L_1^aL_1^b 1 \to -2 L_1^bL_1^a 1 - L_2^{[a,b]}1 \to -2L_1^bL_1^a 1\]

First relation, multiplied by $R_1$, $a>b$:

\[  R_1L_2^ab\to -L_2^aR_1^b\to -L_2aL_1^b1\to - L_1^bL_2^a1+L_3^{[a,b]}1\to 0 \]

\[  R_1L_2^ba\to\dots\to 0 \]

Second relation, multiplied by $R_0$:

\begin{multline*} R_0\partial L_2^a b \to - \partial L_2^a \partial L_1^b 1 \to -\partial L_2^a L_1^b \partial 1 \to -\partial L_1^b L_2^a \partial 1 \\\to -\partial L_1^b \partial L_2^a 1 - 2 \partial L_1^b L_1^a 1 \\\to -2L_1^b\partial L_1^a 1 + 2 L_0^b L_1^a 1 \to -2 L_1^b L_1^a \partial 1 + 2 L_1^{[b,a]}1  \end{multline*}

\begin{multline*}
    R_0L_1^a b + R_0 L_1^b a \to - L_1^a L_1^b \partial 1 - L_1^bL_1^a \partial 1 = -2L_1^bL_1^a\partial 1 - L_2^{[a,b]}\partial 1 \\\to -2L_1^bL_1^a\partial 1 - 2 L_1^{[a,b]}1 = -2L_1^bL_1^a\partial 1 + 2 L_1^{[b,a]}1
\end{multline*}

Second relation, multiplied by $R_1$:

\begin{multline*}
    R_1\partial L_2^a b \to -\partial L_2^a L_1^b 1 - L_2^a L_1^b \partial 1  \to -\partial L_1^b L_2^a 1 - L_1^b L_2^a \partial 1 \to - 2L_1^b L_1^a 1 
\end{multline*}

\begin{multline*}
    R_1L_1^a b + R_1 L_1^b a \to - L_1^a L_1^b 1 - L_1^b L_1^a 1 \to -2L_1^bL_1^a 1 - L_2^{[a,b]}1 \to -2 L_1^bL_1^a 1 
\end{multline*}

Third relation, multiplied by $R_0$, $a>b$:

\begin{multline*}
    R_0L_1^a\partial b \to L_1^a\partial R_0b\to -L_1^a\partial^2L_1^b1\to -L_1^a\partial L_1^b\partial 1 + L_1^a\partial L_0^b 1 \\\to -L_1^a L_1^b \partial^2 1 +L_1^aL_0^b \partial 1 \to -L_1^a L_1^b \partial^2 1 \to -L_1^bL_1^a\partial^2 1 - L_2^{[a,b]}\partial^2 1 \\ \to -L_1^bL_1^a\partial^2 1 - \partial L_2^{[a,b]}\partial 1 - 2L_1^{[a,b]}\partial 1 \\\to -L_1^bL_1^a\partial^2 1 - \partial^2 L_2^{[a,b]}1 - 2\partial L_1^{[a,b]}1 - 2L_1^{[a,b]}\partial 1\to -L_1^bL_1^a\partial^2 1 -4L_1^{[a,b]}\partial 1 
\end{multline*}

\begin{multline*}
    R_0L_1^b\partial a + 3R_0L_0^ab-3R_0L_0^ba-2R_0[a,b] \\ \to - L_1^b\partial^2 L_1^a1 -3L_0^a\partial L_1^b1 + 3L_0^b\partial L_1^a1 + 2 \partial L_1^{[a,b]}1 \\\to -L_1^b\partial L_1^a\partial 1 + L_1^b\partial L_0^a 1 - 3 L_0^aL_1^b\partial 1 + 3 L_0^bL_1^a\partial 1 + 2L_1^{[a,b]}\partial 1 \\ \to -L_1^bL_1^a \partial^2 1 + L_1^bL_0^a\partial 1 -3 L_0^aL_1^b\partial 1 + 3 L_0^bL_1^a\partial 1 + 2L_1^{[a,b]}\partial 1 \\ \to -L_1^bL_1^a \partial^2 1 -3 L_0^aL_1^b\partial 1 + 3 L_0^bL_1^a\partial 1 + 2L_1^{[a,b]}\partial 1 \\\to  -L_1^bL_1^a \partial^2 1 - 3 L_1^{[a,b]}\partial 1 + 3L_1^{[b,a]}\partial 1+2L_1^{[a,b]}\partial 1 \to -L_1^bL_1^a\partial^2 1 -4L_1^{[a,b]}\partial 1 
\end{multline*}

Third relation, multiplied by $R_1$, $a>b$:

\begin{multline*}
    R_1L_1^a\partial b \to - L_1^a\partial L_1^b 1 - L_1^a\partial L_1^b 1 \to -2L_1^aL_1^b\partial 1 \to -2L_1^bL_1^a \partial 1 - 2L_2^{[a,b]}\partial 1 \\\to -2L_1^bL_1^a \partial 1 - 4L_1^{[a,b]}1 
\end{multline*}

\begin{multline*}
    R_1L_1^b\partial a + 3R_1L_0^ab-3R_1L_0^ba-2R_1[a,b] \\ \to -L_1^b\partial L_1^a 1 - L_1^b \partial L_1^a 1 - 3L_0^aL_1^b 1 + 3 L_0^b L_1^a 1 + 2 L_1^{[a,b]}1 \\\to -2L_1^bL_1^a\partial 1 -3 L_1^{[a,b]}1 + 3 L_1^{[b,a]}1 + 2 L_1^{[a,b]}1 \to -2L_1^bL_1^a\partial 1 -4L_1^{[a,b]}1
    \end{multline*}

Let us check \eqref{eq:GSB-Ds}, $s\ge 2$, $a>b$. Left part of relation, multiplied by $R_0$:

\begin{multline*} R_0L_1^a\partial^s b \to -L_1^a\partial^{s+1}L_1^b 1 \to -L_1^a \partial^s L_1^b \partial 1 \to - L_1^a L_1^b \partial^{s+1} 1 - sL_1^a L_0^b \partial^s 1  \\\to -L_1^bL_1^a\partial^{s+1}1 -L_2^{[a,b]}\partial^{s+1} 1 - sL_0^bL_1^a \partial^s 1 - sL_1^{[a,b]}\partial^s 1 \\\to
-L_1^bL_1^a\partial^{s+1}1-L_2^{[a,b]}\partial^{s+1}1-sL_1^{[b,a]}\partial^s 1 - s L_1^{[a,b]}\partial^s 1 \\\to -L_1^bL_1^a\partial^{s+1}1-L_2^{[a,b]}\partial^{s+1}1\to -L_1^bL_1^a\partial^{s+1}1-2(s+1)L_1^{[a,b]}\partial^{s+1}1
\end{multline*}

Right part of relation (one term at a time), multiplied by $R_0$:
\begin{multline*}
R_0L_1^b\partial^s a \to - L_1^b\partial^{s+1} L_1^a 1 \to -L_1^bL_1^a\partial^{s+1}1 + (s+1)L_1^bL_0^a\partial^{s} 1 \\\to -L_1^bL_1^a\partial^{s+1}1 
\end{multline*}

\begin{multline*}
    -R_0(s+2)L_0^b\partial^{s-1}a \to (s+2)L_0^b\partial^{s}L_1^a 1 \\\to (s+2)L_0^bL_1^a \partial^s 1 - s(s+2)L_0^bL_0^a \partial^{s-1} 1\to (s+2) L_1^{[b,a]}\partial^{s}1
\end{multline*}

\[
R_0(s+2)L_0^a\partial^{s-1}b \to -(s+2)L_1^{[a,b]}\partial^s 1
\]

\begin{multline*}
  -2R_0\partial^{s-1}[a,b]\to 2\partial^{s-1}L_1^{[a,b]}1 \to 2L_1^{[a,b]}\partial^{s-1}1 - 2(s-1)L_0^{[a,b]}\partial^{s-2}1 \\\to 2L_1^{[a,b]}\partial^{s-1}1   
\end{multline*}

So, right part goes to 
\begin{multline*} -L_1^bL_1^a\partial^{s+1}1 + (s+2) L_1^{[b,a]}\partial^{s}1 -(s+2)L_1^{[a,b]}\partial^s 1 + 2L_1^{[a,b]}\partial^{s-1}1 \\\to -L_1^bL_1^a\partial^{s+1}1 + 2(s+1) L_1^{[b,a]}\partial^{s}1 = -L_1^bL_1^a\partial^{s+1}1 - 2(s+1) L_1^{[a,b]}\partial^{s}1
\end{multline*}

Left part of relation \eqref{eq:GSB-Ds}, multiplied by $R_1$.

\begin{multline*}
    R_1L_1^a\partial^s b \to -L_1^a\partial^s L_1^b 1 - sL_1^a\partial^{s}L_1^b 1 \to -(s+1)L_1^aL_1^b\partial^s 1 - s L_1^aL_0^b\partial^{s-1}1 \\\to -(s+1)L_1^bL_1^a\partial^s 1 - (s+1)L_2^{[a,b]}\partial^s 1 \to -(s+1)L_1^bL_1^a\partial^s 1 - 2(s+1)s L_1^{[a,b]}\partial^{s-1}1
\end{multline*} 

Right part of relation \eqref{eq:GSB-Ds} (one term at a time), multiplied by $R_0$

\[
    R_1L_1^b\partial^s \to -(s+1)L_1^bL_1^a\partial^s 1 
\]

\begin{multline*}
-R_1(s+2)L_0^b\partial^{s-1}a \to (s+2)L_0^b\partial^{s-1}L_1^a 1 + (s-1)(s+2)L_0^b\partial^{s-1}L_1^a 1 \\\to s(s+2)L_1^{[b,a]}\partial^{s-1}1
\end{multline*}

\[
R_1(s+2)L_0^a\partial^{s-1}b \to -s(s+2)L_1^{[a,b]}\partial^{s-1}1
\]

\[
    -2R_1\partial^{s-1}[a,b]\to 2\partial^{s-1}L_1^{[a,b]}1 + 2(s-1)\partial^{s-1}L_1^{[a,b]} 1 \to 2sL_1^{[a,b]}\partial^{s-1} 1
\]

So, right part goes to 

\[
 -(s+1)L_1^bL_1^a\partial^s 1 -2s(s+1)L_1^{[a,b]}\partial^{s-1}1    
\]

Let us check \eqref{eq:GSB-2.2}, $b\le c < a$. 

\begin{multline*}R_0L_2^aL_2^b c \to - L_2^aL_2^b \partial L_1^c 1 \to -L_2^aL_1^cL_2^b \partial 1 \to -L_1^cL_2^aL_2^b \partial 1 \\\to -2L_1^cL_2^a L_1^b 1 \to 0\end{multline*}

\[R_1L_2^aL_2^b c \to -L_2^aL_2^bL_1^c 1  \to 0 \] 

Let us check \eqref{eq:GSB-1.2}, $b < a \le c$.

\[R_0L_1^aL_2^b c \to -L_1^aL_2^b\partial L_1^c 1 \to -L_1^a\partial L_2^bL_1^c 1 - 2 L_1^aL_1^bL_1^c 1 \to -2L_1^bL_1^cL_1^a 1  \]

\[ R_0L_1^bL_2^c a \to -2 L_1^bL_1^cL_1^a 1\]

\[ R_1L_1^aL_2^b c \to -L_1^aL_2^bL_1^c 1 \to 0\]

\[ R_1L_1^bL_2^c a \to 0\]

 \eqref{eq:GSB-1.2'} is checked similarly.
 
 Let us check \eqref{eq:GSB-1.1'}, $c< a \le b$.  Left part of relation, multiplied by $R_0$
 
\begin{multline*} R_0L_1^aL_1^b c \to - L_1^aL_1^b\partial L_1^c 1 \to -L_1^aL_1^bL_1^c \partial 1 \to -L_1^a L_1^c L_1^b \partial 1 - L_1^aL_2^{[b,c]}\partial 1 \\\to -L_1^cL_1^aL_1^b \partial 1 - L_2^{[a,c]}L_1^b \partial 1 - 2L_1^aL_1^{[b,c]} 1  \\\to -L_1^cL_1^aL_1^b \partial 1 - L_1^bL_2^{[a,c]}\partial 1 -2 L_1^aL_1^{[b,c]} 1 \\\to  -L_1^cL_1^aL_1^b \partial 1 - 2L_1^bL_1^{[a,c]}1 -2L_1^aL_1^{[b,c]}1 \end{multline*}

Right part of relation, multiplied by $R_0$ (one term at a time)

\[
    R_0L_1^cL_1^a b \to -L_1^cL_1^aL_1^b \partial 1
\]

\begin{multline*}
    R_0L_0^bL_2^c a \to -L_0^bL_2^c\partial L_1^a 1 \to -L_0^bL_2^cL_1^a \partial 1 \to -L_0^bL_1^aL_2^c \partial 1 \\\to -2L_0^bL_1^aL_1^c 1 \to -2 L_1^{[b,c]}L_1^a 1 -2 L_1^cL_1^{[b,a]}1
\end{multline*}

\[ -R_0L_0^cL_2^a b \to 2L_0^cL_1^aL_1^b 1 \to 2 L_1^{[c,a]}L_1^b 1 + 2 L_1^a L_1^{[c,b]}1 \]

\[ R_0L_2^c[a,b]\to -L_2^c\partial L_1^{[a,b]}1 \to -L_1^{[a,b]}L_2^c\partial 1 \to -2L_1^{[a,b]}L_1^c 1\]

\[ R_0L_2^a[c,b] \to -2 L_1^{[c,b]}L_1^a 1 \]

So, right part goes to 

\[ -L_1^cL_1^aL_1^b\partial 1 +  2 L_1^{[c,a]}L_1^b 1 + 2 L_1^a L_1^{[c,b]}1 \]

Now, \eqref{eq:GSB-1.1'}, multiplied by $R_1$. Left part:

\[
    R_1L_1^aL_1^b c \to -L_1^aL_1^bL_1^c 1 \to -L_1^aL_1^cL_1^b1 \to -L_1^cL_1^aL_1^b 1
\]

Right part:

\[ R_1L_1^cL_1^a b \to -L_1^cL_1^aL_1^b 1\]

\[ R_1L_0^bL_2^c a \to - L_0^bL_2^cL_1^a 1 \to 0  \]

\[ -R_1L_0^cL_2^a b \to 0\]

\[ R_1L_2^c[a,b] \to - L_2^cL_1^{[a,b]} 1 \to 0\]

\[ R_1L_2^a[c,b]\to 0\]

So, right part goes to \[-L_1^cL_1^aL_1^b 1\]

Let us check \eqref{eq:GSB-1.1}, $a\le c < b$ . Left part, multiplied by $R_0$:

\begin{multline*} R_0L_1^aL_1^b c \to - L_1^aL_1^b\partial L_1^c 1 \to -L_1^aL_1^bL_1^c \partial 1 \to -L_1^a L_1^c L_1^b \partial 1 - L_1^aL_2^{[b,c]}\partial 1 \\\to -L_1^aL_1^cL_1^b \partial 1 - 2L_1^aL_1^{[b,c]} 1  \end{multline*}

Right part, multiplied by $R_0$:

\[ R_0L_1^aL_1^c b \to -L_1^aL_1^cL_1^b \partial 1\]

\begin{multline*} R_0L_0^bL_2^a c \to -L_0^bL_2^aL_1^c \partial 1 \to -L_0^bL_1^cL_2^a \partial 1 \\\to -2L_0^bL_1^aL_1^c 1 \to -2L_1^{[b,a]}L_1^c1 -2L_1^aL_1^{[b,c]}1 \end{multline*}

\[ -R_0 L_0^cL_2^a b \to -2L_0^cL_1^aL_1^b 1 \to 2L_1^{[c,a]}L_1^b 1 + 2L_1^aL_1^{[c,b]}1 \]

\[ R_0L_2^x[y,z]\to - L_2^x L_1^{[y,z]} \partial 1 \to -2L_1^{[y,z]}L_1^x 1 \]

In last relation $x,y,z\in \{a,b,c\}$. So, right part goes to 

\[ -L_1^aL_1^cL_1^b \partial 1  -2L_1^aL_1^{[b,c]}1    \]

Proof for \eqref{eq:GSB-1.1}, multiplied by $R_1$, is pretty similar with \eqref{eq:GSB-1.1'}, since:
\[ R_1L_2^x[y,z]\to -L_2^xL_1^{[y,z]}1 \to 0 \]

Let us check \eqref{eq:GSB-0.1}, $c<b<a$. Left part of the relation, multiplied by $R_0$:

\begin{multline*} R_0L_0^aL_1^b c \to -L_0^aL_1^b L_1^c \partial 1 \to -L_0^aL_1^cL_1^b \partial 1 - 2L_0^aL_1^{[b,c]} 1 \\\to -L_1^{[a,c]}L_1^b-L_1^cL_1^{[a,b]}\partial 1 - 2L_1^{[a,[b,c]]}1 \end{multline*}

Right part of the relation, multiplied by $R_0$:

\[ R_0L_0^aL_1^c b \to  -L_1^{[a,c]}L_1^b 1 - L_1^cL_1^{[a,b]}1 \]

\[ R_0L_0^bL_1^a c -R_0L_0^bL_1^ca \to L_0^bL_1^cL_1^a\partial 1 + 2L_0^bL_1^{[a,c]} 1 - L_0^bL_1^cL_1^a\partial 1 \to 2L_1^{[b,[a,c]}1    \]

\[ R_0L_0^cL_1^b a -R_0L_0^cL_1^ab \to -L_0^cL_1^bL_1^a1 \partial 1 + L_0^cL_1^b L_1^a1 \partial 1 + 2L_0^cL_1^{[a,b]}1 \to 2L_1^{[c,[a,b]]}1 \]

\[ R_0L_1^{[c,a]}b+R_0L_1^{[a,b]}c+R_0L_1^{[b,c]}a \to -L_1^{[c,a]}L_1^{b}\partial 1 -L_1^{[a,b]}L_1^c \partial 1 - L_1^{[b,c]}L_1^a \partial 1  \]

\[ -R_0L_1^c[a,b]-R_0L_1^a[b,c]-R_0L_1^b[c,a]\to L_1^cL_1^{[a,b]}\partial 1+L_1^aL_1^{[b,c]}\partial 1+L_1^bL_1^{[c,a]}\partial 1\]

So, the composition is 0 by Jacobi identity.

Left part of the relation, multiplied by $R_1$:

\[ R_1L_0^aL_1^b c \to -L_0^aL_1^bL_1^c 1 \to -L_0^aL_1^cL_1^b 1 \to -L_1^{[a,c]}L_1^b 1 - L_1^cL_1^{[a,b]}1  \]

Right part of the relation, multiplied by $R_1$:

\[ R_1L_0^aL_1^c b \to  -L_1^{[a,c]}L_1^b 1 - L_1^cL_1^{[a,b]}1 \]

\[ R_1L_0^bL_1^a c -R_1L_0^bL_1^ca \to  -L_0^bL_1^aL_1^c 1 + L_0^bL_1^cL_1^a 1  \to 0 \]

\[ R_1L_0^cL_1^b a -R_1L_0^cL_1^ab \to -L_0^cL_1^bL_1^a 1 + L_0^cL_1^aL_1^b 1 \to 0\]

\[ R_1L_1^{[c,a]}b+R_1L_1^{[a,b]}c+R_1L_1^{[b,c]}a \to -L_1^{[c,a]}L_1^{b}1 -L_1^{[a,b]}L_1^c 1 - L_1^{[b,c]}L_1^a 1  \]

\[ -R_1L_1^c[a,b]-R_1L_1^a[b,c]-R_1L_1^b[c,a]\to L_1^cL_1^{[a,b]}+L_1^aL_1^{[b,c]}1+L_1^bL_1^{[c,a]}1\]

So, the composition is zero. This completes the proof. 
\end{proof}

\begin{corollary}
The linear basis of $H\otimes P(\mathfrak g)$ 
as of $H$-module 
consists of the words $L_1^{b_1}\dots L_k^{b_k}1$, 
$b_1\le \dots \le b_k$.
\end{corollary}

Indeed, the reduced words ending with $1$ 
do not contain $L_0^a$ or $L_n^a$ for $n\ge 2$. 
The result agrees with the classical Poincar\'e--Birkhoff--Witt Theorem for $P(\mathfrak g)
=\mathrm{gr}\,U(\mathfrak g)$. 

Being applied to the case when $\mathfrak g$ is a free Lie algebra 
$\Lie (G)$
generated by a set $G$, Theorem~\ref{thm:PBW} provides us a setting 
for calcuating a Gr\"obner--Shirshov basis in the free Poisson algebra.

Namely, every element from the free Poisson algebra $\Pois (G) =P(\Lie (G))$
may be presented as a rewriting rule in the free $A(X)$-module 
generated by a single element 1, where $X$ is the set of nonassociative Lyndon--Shirshov words in the alphabet~$G$.

\begin{corollary}[Composition-Diamond Lemma for Poisson algebras]
If a set $S\subset \Pois (G)$ defines a set of rewriting rules that have no nontrivial compositions  
with \eqref{eq:U3-rulesA} and \eqref{eq:Leibniz}
then $S$ along with the rules mentioned in Theorem~\ref{thm:PBW} is a Gr\"obner--Shirshov basis of $H\otimes \Pois(G\mid S)$ 
considered as a conformal module over $U(\Cur\Lie(G), N=3)$.
\end{corollary}




\end{document}